\documentclass{amsart}
\usepackage{amsthm,amsmath,amssymb,amstext,latexsym,amsfonts,graphicx}
\usepackage[latin1]{inputenc}
\usepackage[T1]{fontenc}
\usepackage[dvips]{epsfig}

\reversemarginpar

\def\esp#1#2{\mathbb{E}_{#1}\left[ #2 \right]}
\def\espcond#1#2{\mathbb{E} \left[ {#1} \,|\, #2 \right]}

\def\R{{\mathbb R}}
\def\car{{\mathbf 1}}
\def\d{\text{ d}}
\def\T{{\mathcal T}}

\newcommand{\Lip}{\operatorname{Lip}}
\newcommand{\Dom}{\operatorname{Dom}}
\newcommand{\id}{\operatorname{Id}}
\def\GD{\nabla^\sharp}
\def\GC{\nabla^c}

\def\LD{{\mathcal L}^\sharp}
\def\LC{{\mathcal L}^c}

\def\PD{P^\sharp}
\def\PC{P^c}

\def\DD{\delta^\sharp}
\def\DC{\delta^c}

\newtheorem{theo}{Theorem}
\newtheorem{de}{Definition}
\newtheorem{remar}{Remark}

\newtheorem{lem}{Lemma}
\newtheorem{hyp}{Hypothesis}

\begin{document}

\title{Rubinstein distance on configurations spaces}
\author{L. Decreusefond} \address{
  ENST - CNRS UMR 5141\\
  Département Informatique et Réseaux\\
  46, rue Barrault, 75634 Paris Cedex 13, France }
\email{Laurent.Decreusefond@enst.fr} \author{N. Savy} \address{
  Institut de Mathématiques de Toulouse\\
  Laboratoire de Statistique et Probabilités\\
  Université Paul Sabatier\\
  118, route de Narbonne, 31062 Toulouse Cedex 9, France }
\email{savy@math.ups-tlse.fr} \date{\today}

\begin{abstract} 
  By a method inspired of the Stein's method, we derive an upper-bound
  of the Rubinstein distance between two absolutely continuous
  probability measures on configurations space. As an application, we show that the best way to  approximate  a Modulated Poisson Process (see below for the definition) by a Poisson process is to equate their  intensity.
\end{abstract}
\maketitle

\section{Introduction}

According to the Kantorovitch approach, the optimal transportation
problem or Monge-Kantorovitch problem (MKP for short) reads as
follows: given two probability measures $\mu$ and $\nu$ on a Polish
space $X$ and a cost function $c$ on $X\times X$, does there exist a
probability measure $\gamma$ on $X\times X$ which minimizes $\int c \d
\beta$ among all probability measures $\beta$ on $X\times X$ with
first (respectively second) marginal $\mu$ (respectively $\nu$)~? The
first step is to determine whether or not there exists a probability
measure $\gamma$ such that $\int c\d \gamma$ is finite. In the solved
cases, a few criterion are known. The oldest one (see
\cite{MR99k:28006}), for quadratic cost, stands that such a measure
exists provided that $\mu$ and $\nu$ have finite second moments. Still
for the quadratic cost, if $\mu$ is a Gaussian measure on a finite or
infinite dimensional space and $\nu =L \mu$, then the distance is
finite whenever $L$ has a finite Boltzmann entropy
\cite{MR2036490}. In the same reference, a bound may also be found for
the Rubinstein distance, i.e., when $c$ is a distance function. We are
interested in the evaluation of the Rubinstein distance on
configurations space, i.e., for locally finite point processes. The
first point to be stressed is that we have several reasonable distance
between configurations. To name only the two we will investigate here:
there is the total variation distance when configurations are viewed
as atomic measures, and there also is a distance with a greater
geometric flavor, which is defined in \eqref{eq:7}. To these two
distances correspond two different notions of Lipschitz continuous
functions. This is of some great importance since the
Kantorovitch-Rubinstein duality allows us to write the Rubinstein
distance as a maximization problem on the set of Lipschitz functions,
see \eqref{eq:9}. Moreover, it is well known in finite dimension that
Lipschitz functions are ``almost'' differentiable and that their
differential is bounded. It turns out that the two usual gradients
introduced on configurations space (see
\cite{MR99d:58179},\cite{p94a}) are the good tools to obtain an analog
result on configurations space. Once we have a gradient, we usually
introduce the divergence (as its adjoint) and then a number
operator (as the composition of the divergence and the gradient),
hence an Ornstein-Uhlenbeck (see \cite{bouleau-hirsch}) semi-group. At
this point, it is useful to invoke the Stein's method which is a very
efficient tool to obtain stochastic bounds notably for point processes
(see \cite{X04,MR2002k:60039,MR2000h:60029,MR1190904}). In essence,
Stein's method compares the expectations of two distinct random
variables by ``embedding'' them in the evolution of an ergodic Markov
process and by then looking backward at the evolution of this process,
from infinity to time $0$. With very sketchy notations, imagine that
we have two smooth functions $\alpha$ and $\beta$ with the same limit
at infinity, then $\alpha(0)-\beta(0)$ may be evaluated by computing
\begin{equation*}
  \int_{0}^{\infty}|\alpha^{\prime}(s)-\beta^{\prime}(s)|\d s.
\end{equation*}
Controlling the difference of derivatives yields to a bound on the
difference at time $0$.  This is exactly the principle at work in
Theorem \ref{T:E1}. This method reminds of the so-called semi-group
method often used in proofs of concentration inequality
\cite{MR1767995}.
In Section 2, we present the generic principle of our method and we
apply it to the different distances on configurations spaces in
Section 3. Section 4 is devoted to the application of these results on
the approximation of Markov modulated Poisson Process (MMPP) by a
Poisson process. The motivation for this part comes from queueing
theory where MMPP are widely used because of their versatility, useful
to model a wide range of real systems
\cite{Fischer:1993qy,Meier-Hellstern:1989lr,MR95k:60241} and of the
persistence of their Markovianity. Unfortunately, these processes are
affected by the curse of dimensionality: it is often the case that we
must invert linear systems with a so huge number of variables it
becomes unfeasible. It is thus of crucial importance to reduce the
cardinality of state space. The extreme situation is when this space
is reduced to one point, i.e., when an MMPP is a Poisson process. We
find by the method developed in the beginning of this paper, a bound
on the Rubinstein distance between an MMPP and a Poisson
process. Optimizing this bound yields to the well known rule of thumb
which consists in taking as the optimal intensity, the average
intensity of the MMPP. Our result is not then astoundingly original
but it shows that by proceeding along this line, we can control the
error for any functional such as loss probability or others.

\section{Generic scheme}
Let $X$ be a Polish space and $d$ a lower-semi-continuous distance
function on $X\times X$, which does not necessarily generate the
topology on $X$. We will denote by $d-\Lip_m$ the set of Lipschitz
continuous $F$ from $X$ to $\R$ with Lipschitz constant $m$:
\begin{equation*}
  |F(x)-F(y)| \le \, m\, d(x,y),
\end{equation*}
for any $(x,\, y) \in X^2$. For two probability measures $\mu$ and
$\nu$ on $X$, the optimal transportation problem associated to $d$
consists in evaluating
\begin{equation*}
  \T_d(\mu,\, \nu)=  \inf_{\gamma \in \Sigma(\mu,\, \nu)} \int_{X\times X} d(x,y)\, \d\gamma(x,\, y),
\end{equation*}
where $\Sigma(\mu, \, \nu)$ is the set of probability measures on
$X\times X$ with first (respectively second) marginal $\mu$
(respectively $\nu$). According to \cite{MR622552,MR1964483}, this
minimum is equal to
\begin{equation}\label{eq:9}
  \T_d(\mu,\, \nu)~=~  \sup_{F \in d-\Lip_1}\int F \d (\mu-\nu).
\end{equation}
We consider the situation where $X=\Gamma_{\Lambda}$ is the configurations space on a
Lusin space $\Lambda$, i.e.,
\begin{displaymath}
  \Gamma_\Lambda=\{\eta \subset \Lambda; \ \eta\cap K \text{ is a finite set
    for every compact } K\subset \Lambda \}.
\end{displaymath}
We identify $\eta\in\Gamma_\Lambda$ and the positive Radon measure
$\sum_{x\in\eta} \varepsilon_x.$ Throughout this paper,
$\Gamma_\Lambda$ is endowed with the vague topology, i.e., the weakest
topology such that for all $f\in {\mathcal C}_0$ (continuous with
compact support on $\Lambda$), the maps
\begin{displaymath}
  \eta \mapsto \int_{\Lambda} f\d \eta=\sum_{x\in\eta} f(x)
\end{displaymath}
are continuous. When $f$ is the indicator function of a subset $B,$ we
will use the shorter notation $\eta( B)$ to denote the integral of
$\car_B$ with respect to $\eta$.  We denote by ${\mathcal
  B}(\Gamma_{\Lambda})$ the corresponding Borel $\sigma$-algebra. The
probability space under consideration will then be $(\Gamma_\Lambda,\  {\mathcal
  B}(\Gamma_{\Lambda}),\, \mu)$. We
need some additional structure.
\begin{hyp} \label{H:1} Assume now that we have~:
  \begin{itemize}
  \item a  kernel $Q$ on $X\times \Lambda$, i.e., such that $Q(A,.)$
    is measurable as function on $\Lambda$ for any $A\in  {\mathcal
  B}(\Gamma_{\Lambda})$ and $Q(.,\, s)$ is a $\sigma-$finite measure
on $X$ for any $s\in \Lambda,$
  \item a map $\nabla$, defined on a subset $\Dom \nabla$ of $L^2(\mu),$
  such that, for any $F\in\Dom \nabla$,
$$
\int_{X} \int_{\Lambda} |\nabla_s F|^2 Q(\omega, \d s) \d \mu(\omega) < +
\infty.
$$
  \end{itemize}
\end{hyp}
We say that a process $u(\omega,\, s)$ belongs to $\Dom \delta$
whenever, there exists a constant $c$ independent of $F$ such that for
any $F\in \Dom \nabla$,
\begin{equation*}
  |  \esp{}{\int_{\Lambda} \nabla_{s}F \, u(s)\ Q(\omega, \d s)}|\le c \| F\|_{L^{2}(\Omega)}.
\end{equation*}
For such a process $u$, we define $\delta u$ by
\begin{equation} \label{E:DEFQ} \int_X \int_{\Lambda} \nabla_s F(\omega) \,
  u(\omega,\, s)\
  Q(\omega, \d s) \d \mu(\omega) = \int_X F \, \delta u \d \mu.
\end{equation}
  \begin{de}
\label{def:rademacher}
We say that $(\nabla,\, Q)$ has the Rademacher property whenever     $F\in d-\Lip_1$ implies $\nabla F \in \Dom
\nabla$ and 
\begin{equation} \label{eq:4} | \nabla_s F| \le 1,\ Q(\omega, .) d \mu\text{-almost-surely}.
\end{equation}
  \end{de}
Consider for $F\in \Dom \nabla$, the (formal)
equation \begin{equation}\label{eq:1} \frac{\d}{\d t}X_t = - \delta
  \nabla X_t, \quad X_0=F.
\end{equation}
If this equation has one and only one solution for each $F\in \Dom
\nabla$, we then have a $\mu$-self-adjoint semi-group $(P_t, \, t\ge
0)$, usually called the  Ornstein-Uhlenbeck semi-group:
\begin{math}
  P_tF=X_t
\end{math} where $X_t$ is the solution of \eqref{eq:1}.
\begin{de}
  The Ornstein-Uhlenbeck is said to be ergodic whenever
  \begin{math}
    \lim_{t\to +\infty}P_tF=\int_X F\d \mu.
  \end{math}
\end{de}
\begin{theo} \label{T:E1} Assume that  hypothesis \ref{H:1} holds. Let
  $\nu$ be another probability measure on $X$ absolutely continuous
  with respect to $\mu$. We denote by $L$ the Radon-Nikodym derivative
  of $\nu$ with respect to $\mu$. If $(\nabla, \, Q)$ has the
  Rademacher property and if
  the Ornstein-Uhlenbeck semi-group is ergodic then:
  \begin{equation}\label{eq:8}
    \T_d(\mu, \, \nu)\le \int_{X \times \Lambda}\, \int_0^{+\infty}  \left| \nabla_s P_t L\right|  \d t \, Q(\omega, \d s) \d \mu(\omega).
  \end{equation}
\end{theo}
\begin{proof}
According to the fundamental Lemma of analysis,
\begin{align*}
  \int_X F\d \mu- F& =\int_0^{+\infty} \frac{\d}{\d t}P_tF \d t,\\
  & = \int_0^{+\infty}  \delta\nabla P_t F \d t,\\
  &= \int_0^{+\infty} P_t \delta\nabla F \d t.
\end{align*}
Since $d\nu/d\mu=L$,
\begin{equation}
  \label{eq:2}
\begin{split}
  \int_X F\d \mu -\int_X F\d \nu &= \int_X \left(\int_X F\d \mu-F\right)\d \nu\\
  & = \int_X \left(\int_0^{+\infty} P_t \delta\nabla F \d t\right)\d
  \nu \notag\\
  &=\int_X \int_0^{+\infty}  P_t  \delta\nabla F L \d t \d \mu \notag\\
  &= \int_{X \times \Lambda} \int_0^{+\infty} \nabla_s F \ \nabla_s P_t L \d t
  \, Q(\omega, \d s) \d \mu(\omega).
\end{split}    
\end{equation}
Since $(\nabla,\, Q)$ has the Rademacher property, we have  an
 $L^\infty-L^1$ bound which yields to~(\ref{eq:8}).
\end{proof}
\section{Instantiations on Poisson space} \label{sec:instantiations}

 Let
$\rho$ be a $\sigma$-finite measure on $\Lambda$ and assume that $\mu$
is the Poisson measure of intensity $\rho$, i.e., the probability
measure on $\Gamma_\Lambda$ fully characterized by
\begin{equation*}
  \esp{}{\exp(\int_{\Lambda} f \d \eta)}=\exp(\int_{\Lambda} \left( e^{f(s)}-1\right) \d \rho(s)\, ).
\end{equation*}
\subsection{Discrete gradient on Poisson
  space} \label{sec:discr-grad-poiss}
For $F: X\to \R$, the discrete gradient of $F$, denoted by $\GD F$, is
defined by
\begin{equation*}
  \GD_s F(\eta)= F(\eta + \varepsilon_s)-F(\eta).
\end{equation*}
We set $Q(\omega,\, ds)=\d\rho(s)$ so that $\Dom \GD$ is defined as
the set of functionals such that 
\begin{equation*}
\esp{}{  \int_\Lambda |\GD_s F|^2 \d\rho(s)}<+\infty.
\end{equation*}
We denote by $\DD$ its adjoint in the sense of (\ref{E:DEFQ}).
The $n$-th iterated integral of a symmetric function $f$ from
$\Lambda^n$ to $\R$ is defined as
\begin{equation*}
  J_n(f)=n!\underset{0\le s_1 <s_2<\cdots<s_n}{\int\ldots
    \int}f(s_1,\cdots,s_n)\, \d(\eta-\rho)(s_1)\ldots\d(\eta-\rho)(s_n).
\end{equation*}
For a general function $f$,
\begin{equation*}
  J_n(f)=\sum_{\sigma\in {\mathcal S}_n}\underset{0\le s_1 <s_2<\cdots<s_n}{\int\ldots
    \int}f(s_{\sigma(1)},\cdots,s_{\sigma(n)})\, \d(\eta-\rho)(s_1)\ldots\d(\eta-\rho)(s_n).
\end{equation*}
It is well known \cite{Ruiz1985,p94a} that any square integrable
functional on $\Gamma_\Lambda$ can be written as
\begin{equation*}
  F=\sum_{n=0}^{+\infty} J_n(f_n),
\end{equation*}
where for any integer $n$, $f_n$ is symmetric and belongs to
$L^2(\rho^{\otimes (n)})$ and that
\begin{align*}
  \GD_s F(\eta) = \sum_{n=1}^{+\infty} n J_{n-1}(f_n(.,\, s)).
\end{align*}
Moreover, the Ornstein-Uhlenbeck semi-group operates on chaos as:
\begin{equation}
  \label{eq:5}
  \PD_t F= \sum_{n=1}^{+\infty} e^{-nt} J_n(f_n)
\end{equation}
From  \eqref{eq:5} and by dominated convergence, it is then easily
seen that $\PD$ is ergodic.
We now choose the total variation as the distance of interest on
$\Gamma_\Lambda$, i.e.,
\begin{equation*}
  d_1(\eta,\, \omega)=2 \sup_{A\in \Lambda}\left| \eta(A)-\omega(A)\right|
\end{equation*}
\begin{lem}\label{lem:lip_discrete}
For the distance  $d_1$ on $X$, $(\GD,\, Q)$ has the Rademacher property.
\end{lem}
\begin{proof} 
  Consider $F \in d_1-\Lip_1$, by the very definition of the gradient:
  \begin{align*}
    \left| \GD_s F(\eta) \right| &\le  |F(\eta+\varepsilon_s)-F(\eta)| \\
    &\le d_1(\eta+\varepsilon_s, \, \eta) ~=~1.
  \end{align*}
   In the converse direction, consider $\omega$ and $\eta$ be two locally
 finite but not finite configurations. If $d_1(\omega,\,
 \eta)=+\infty$, there is nothing to prove. If $d_1(\omega,\,
 \eta)$ is finite, $\omega\Delta \eta$ and $\eta\Delta
    \omega$ are finite,  where $\omega\Delta \eta=\omega \backslash (\omega \cap
  \eta)$.
  Since $|\GD_s F(\eta)|\le 1$, we get:
  \begin{align*}
    |F(\eta)-F(\omega)|&\le |F(\eta \cap \omega \, \cup \, \eta\Delta
    \omega)-F(\eta\cap \omega)| +   |F(\eta \cap \omega \, \cup \, \omega\Delta
    \eta)-F(\eta\cap \omega)|\\
&\le (\eta\Delta
    \omega) (\Lambda) + (\omega\Delta \eta)(\Lambda)\\
&\le 2\max((\eta\Delta
    \omega) (\Lambda),\, (\omega\Delta \eta)(\Lambda))\\
& = d_1(\eta,\, \omega).
  \end{align*}
 The Rademacher property is then established for $(\GD, Q)$.
\end{proof}
\begin{theo} \label{T:RF1} Let $\mu$ and $\nu$ two probability
  measures on $\Gamma_\Lambda$ such that $\d \nu =L\d \mu$. We have,
  \begin{equation*}
    \T_{d_1}(\mu, \, \nu)\le \esp{}{ \int_{\Lambda} | (\id+\LD)^{-1}\GD_s L | \d\rho(s)},
  \end{equation*}
where $\LD=\DD\GD$.
\end{theo}
\begin{proof}
  It is  easily seen using chaos decomposition that
  \begin{equation*}
    \GD_{s}\PD_tF = e^{-t}\PD_t\GD_sF \quad \text{for all } s \in
    \lambda, \, \text{ for all } t \in \R^+
  \end{equation*}
  and it is a general property of semi-groups and their generator that
  \begin{equation*}
    \int_{0}^{+\infty}e^{-t}\PD_{t}F \d t=(\id + \LD)^{-1}F,
  \end{equation*}
for any $F\, : \, \Omega \to \R$.
We then infer from (\ref{eq:8}) that
\begin{align*}
\T_{d_1}(\mu, \, \nu)&\le
\esp{}{\int_{\Lambda} |\int_0^{+\infty}
e^{-t}\PD_{t}\GD_{s}L \d t| \d \rho(s)}\\
&=\esp{}{ \int_{\Lambda}|(\id +\LD)^{-1}(\GD_{s}L)|\d\rho(s)}
\end{align*}
\end{proof}
\begin{remar}
  Note that the very analog of this inequality on Wiener space was proved by a
  different though related way in \cite{FU04}.
\end{remar}
\subsection{Derivation on Poisson space}
In this section we introduce another stochastic gradient on
$\Gamma_\Lambda$ which is a derivation -- see \cite{MR99d:58179}.  Let
$V(\Lambda)$ be the set of ${\mathcal C}^\infty$ vector fields on
$\Lambda$ and $V_0(\Lambda)\subset V(\Lambda),$ the subset consisting
of all vector fields with compact support.  For $v\in V_0(\Lambda),$
for any $x\in \Lambda,$ the curve
\begin{displaymath}
  t\mapsto {\mathcal V}_t^v(x)\in \Lambda
\end{displaymath}
is defined as the solution of the following Cauchy problem
\begin{equation}
  \label{eq:D2}
  \begin{cases}
    \dfrac{d}{dt}{\mathcal V}_t^v(x)&=v({\mathcal V}_t^v(x)),\\
    {\mathcal V}_0^v(x)&=x.
  \end{cases}
\end{equation}
The associated flow $({\mathcal V}_t^v,\, t\in \R)$ induces a curve
$({\mathcal V}_t^v)^*\eta=\eta\circ ({\mathcal V}_t^v)^{-1}$, $t\in
\R$, on $\Gamma_\Lambda$: If $\eta=\sum_{x\in \eta}\varepsilon_x$ then
$({\mathcal V}_t^v)^*\eta=\sum_{x\in \eta}\varepsilon_{{\mathcal
    V}_t^v(x)}.$ We are then in position to define the notion of
differentiability on $\Gamma_\Lambda$.  A measurable function $F\, :
\, \Gamma_\Lambda \to \R$ is said to be differentiable if for any
$v\in V_0(\Lambda)$, the following limit exists:
\begin{equation*}
  \lim_{t\to 0} t^{-1}\left( F({\mathcal V}_t^v(\eta))-F(\eta)\right).
\end{equation*}
We then denote $\GC_vF(\eta)$ the preceding limit.  We denote by
$\GC_sF$ we corresponding gradient. It verifies~:
\begin{equation*}
  \GC_v F (\omega) =\int_\Lambda  \GC_s F (\omega) v(s) \d \omega(s).
\end{equation*}
The square norm of  $\GC F$ is given by
\begin{equation*}
  \int_{\Lambda} \GC_{s}F \d \omega(s),
\end{equation*}
so that we are in the framework of Hypothesis 1 if we take 
\begin{equation*}
  Q(\omega,\, ds)=d \omega(s)=\sum_{x\in \omega }\varepsilon_{x}(d s),
\end{equation*}
where $\varepsilon_{a}$ is the Dirac mass in $a$.
For a random variable $F : \Gamma_\Lambda \to \R$, and random process
$u : \Gamma_\Lambda \times \Lambda \to \R$, we define the adjoint
operator of $\GC$ denoted by $\DC$ by:
\begin{equation*}
  \esp{}{ < \GC F \, , \, u >_{L^2(\rho)} } =\esp{}{ F \, \delta^2u },
\end{equation*}
provided both sides exist, i.e.,
\begin{equation*}
  \left|\esp{}{ < \GC F \, , \, u >_{L^2(\rho)} }\right|\le c
  \Vert F\Vert _{L^2}^2.
\end{equation*}
Consider now $\LC =  \DC\GC$ and the associated semi-group semigroup
$\{ \PC_t, t \in \R\}$. The distance of interest is here the
Wassertein's distance (see \cite{Decreusefond:2006cv,MR1730565}):
\begin{equation}\label{eq:7}
  d_2(\eta_1,\eta_2)=\left[\inf \left\{ \int d_0(x,y) d\beta(x,y),\
      \beta\in \Gamma_{\eta_1,\eta_2}\right\}\right]^{1/2},
\end{equation}
where $\Gamma_{\eta_1,\eta_2}$ denotes the set of
$\beta\in\Gamma_{\Lambda\times \Lambda}$ having marginals $\eta_1$ and
$\eta_2.$ The ergodicity of $\PC$ is proved in \cite{MR99d:58179} and
the Rademacher property is the object of \cite{MR1730565}.

On the other hand, there is no known commutation relationships between
$\GC$ and $\PC_t$ hence theorem \ref{T:E1} entails that
\begin{theo} \label{T:RF2} Let $\mu$ and $\nu$ two probability
  measures on $\Gamma_\Lambda$ such that $\d \nu =L\d \mu$. We have,
  \begin{equation*}
    \T_{d_2}(\mu, \, \nu) \le \esp{}{   \int_0^{+\infty}  \left| \GC \PC_t L \right|  \d t }.
  \end{equation*}
\end{theo}

\section{Applications}
\subsection{Distance between two Poisson processes} \label{S-DPP}
\begin{theo}
  Consider $\mu$ and $\nu$ two Poisson probability measures on
  $\Lambda\subset \R^{n}$. The intensity of $\mu$ is the Lebesgue
  measure on $\Lambda$ and that of $\nu$ is $h(s)d s $ with $h \in L^2(\Lambda)$ deterministic.
  Then, we have the following bound~:
  \begin{equation} \label{E-RES} \T_{d_1}(\mu, \, \nu)~\leq~C \,
    ||h||_{L^{1}(\Lambda)} \, \exp(\frac{1}{2} \,
    ||h-1||_{L^{2}(\Lambda)}^2).
  \end{equation}
\end{theo}
\begin{proof}
  According to \cite{jacod},
  \begin{equation*}
    L(\omega)=\frac{\text{d}\nu}{\text{d}\mu}(\omega)=\exp\Bigl(\int_{\Lambda}
    \ln h(s) \d \omega(s)-\int_{\Lambda} (h(s)-1)\d s\Bigr).
  \end{equation*}
  By definition,
  \begin{equation*}
    \GD_{s}L(\omega)=L(\omega+\varepsilon_{s})-L(\omega)=h(s)L.
  \end{equation*}
  Hence, according to Theorem \ref{T:RF1},
  \begin{equation*}
    \T_{d_1}(\mu, \, \nu)~\leq~C \|h\|_{L^{1}(\Lambda)}\esp{}{|(I+\LD)^{-1}L|}.  
  \end{equation*}
  It is then well known, using for instance the chaos decomposition
  that
  \begin{equation*}
    \esp{}{|(I+\LD)^{-1}L|}\le\esp{}{L^{2}}^{1/2}
\end{equation*}
and
\begin{align*}
\esp{}{L^{2}}  &=\esp{}{\exp\Bigl(2\int_{\Lambda}
      \ln h(s) \d \omega(s)-2\int_{\Lambda} (h(s)-1)\d s\Bigr)}\\
    &=\esp{}{\exp\Bigl(\int_{\Lambda} \ln h^{2}(s) \d \omega(s)-\int_{\Lambda}
      (h^{2}(s)-1)\d
      s\Bigr)}\exp(\int_{\Lambda}(h(s)-1)^{2}\d s)\\
    &=\exp(\int_{\Lambda}(h(s)-1)^{2}\d s).
  \end{align*}
  The proof is thus complete.
\end{proof}

\subsection{Distance between a Poisson process and a Markov modulated
  Poisson process}

In this section we calculate a bound for the distance between a Poisson Process and an Markov modulated  Poisson process (MMPP for short).
\begin{de}
Consider $J$ an irreducible continuous time Markov chain with finite
state space. We denote by
 $m_J$ the finite number of states of $J$,
 $Q_J$ the infinitesimal generator of $J$,
 $\pi_J$ the stationary vector of $Q_J$. We assume also that we are
 given $(\lambda_{1},\cdots,\, \lambda_{m_{J}})$, $m_{J}$ non-negative real numbers. We denote by $\Psi$ the map
 which sends $i\in \{1,\cdots,\, m_{J}\}$ to $\lambda_{i}.$
An $MMPP(J, \Psi)$ is a point process the intensity of which is
given by $\Psi(J_{s})ds$. This means  that when $J$ is in state (
called a phase) $k$ $(1 \leq k \leq m_J)$ then the arrivals occurs according to a Poisson process of rate $\lambda_k$.
\end{de}
A detailed description of the MMPP with an emphasis on applicability to modeling is given in \cite{Meier-Hellstern:1989lr} and references therein.
\begin{theo}
Let $\mu$ be a Poisson process of intensity $\lambda$ and $\nu$ be a $MMPP(J, \Psi)$ then for a given $T \in \R^{+}$,
\begin{equation} \label{E:DISTPOIMMPP}
\T_{d_1}(\mu_T, \nu_T) \leq \esp{}{\int_0^T \left| \frac{\Psi(J_s)}{\lambda} \right| \lambda \d s 
\exp \left[ \frac{1}{2} \int_0^T \left| \frac{\Psi(J_s)}{\lambda} - 1 \right|^2 \lambda \d s \right]}
\end{equation}
\end{theo}
\begin{proof}
By the Girsanov formula we have:
\begin{equation*}
    L_T(\omega)=\frac{\text{d}\nu_T}{\text{d}\mu_T}(\omega)=\exp\Bigl(\int_0^T
    \ln \frac{\Psi(J_s)}{\lambda} \d \omega(s)-\int_0^T (\frac{\Psi(J_s)}{\lambda}-1) \lambda \d s\Bigr).
  \end{equation*}
Consider  
\begin{math}
  \mathfrak{F}_{J}^{T}=\sigma\{J_{s}, s\le T\},
\end{math}
the history of $J$ up to time $T$.
 Now, 
\begin{align*}
\T_{d_1}(\mu_T, \nu_T) 
&= \sup_{F \in d_1-\Lip_1} \left( \esp{\mu_T}{F (L_T-1)}\right)\\
&= \sup_{F \in d_1-\Lip_1} \left( \esp{\mu_T}{ \espcond{F (L_T-1)}{\mathfrak{F}_{J}^{T}}} \right)\\
&\le \esp{\mu_T}{ \sup_{F \in d_1-\Lip_1} \left( \espcond{F (L_T-1)}{\mathfrak{F}_{J}^{T} }\right)}\\
\end{align*}
But conditioning on $J$, the intensity is deterministic so according
to Theorem \ref{E-RES},
\begin{align*}
\T_{d_1}(\mu_T, \nu_T) 
&= \esp{}{\int_0^T \left| \frac{\Psi(J_s)}{\lambda} \right| \lambda \d s \,
\exp \left[ \frac{1}{2} \int_0^T \left| \frac{\Psi(J_s)}{\lambda} - 1 \right|^2 \lambda \d s \right]},
\end{align*}
which ends the proof.
\end{proof}
One can then try to determinate the nearest Poisson process to a given
MMPP. For, we seek for $\lambda_{opt}$ such that the upper bound of \eqref{E:DISTPOIMMPP} is minimal.
It is clear that this minimum is directed by the exponential part of
the expression. It is thus enough to minimize
$$
\Phi_T(\lambda) = \int_0^T \left| \frac{\Psi(J_s)}{\lambda} - 1 \right|^2 \lambda \d s
$$
It is well known, that for large $T$, we have~:
\begin{equation*}
\Phi_T(\lambda) \underset{T \to \infty}{\thicksim} \lambda
T \ \sum_{i=1}^m \left| \frac{\lambda_i}{\lambda} - 1 \right|^2 \pi(i)
\end{equation*}
which is minimal for $\lambda = \lambda_{opt} = \sum_{i=1}^m \, \lambda_i \, \pi(i)$.
Moreover, 
\begin{align*}
\sum_{i=1}^m \left| \frac{\lambda_i}{\lambda_{opt}} - 1 \right|^2 \lambda_{opt} \pi(i)
&=  \sum_{i=1}^m  \frac{(\lambda_i - \lambda_{opt})^2}{\lambda_{opt}} \pi(i)\\
&= \frac{V_{opt}}{\lambda_{opt}}
\end{align*}
which is known in queueing theory as the burstiness of the
MMPP. Finally the distance between an MMPP and the Poisson process of
intensity equal the mean arrival rate of the MMPP is bounded by 
\begin{equation*}
  \lambda_{opt}T\,\exp\Bigl(\frac{V_{opt}}{2\lambda_{opt}}\, T\Bigr).
\end{equation*}
In queueing theory, the choice of 
$\lambda_{opt}$ as $\sum_{i=1}^m \, \lambda_i \, \pi(i)$ is imposed by
``load'' conservation: one can only compare queueing systems with the
same load, i.e., the load (or traffic) is defined as the product of
the mean arrival rate and of the mean service time. Our result shows
that this choice is likely to be the optimal one. Moreover, we are now
in position to evaluate precisely the error due to this approximation. 
Our bound gives a qualitative basis for the experimental rule that not
only the load was important to evaluate performance of queueing system but also the
so-called burstiness was to be taken into account.

\end{document}